\documentclass[reqno,dvipdfmx]{amsart} 

\usepackage{mathrsfs} 
\usepackage{amsmath} 
\usepackage{amssymb} 
\usepackage{bm} 
\usepackage{cite} 
\usepackage{latexsym} 
\usepackage{graphicx} 

\usepackage{color} 
\usepackage{comment} 

\theoremstyle{plain} 
\newtheorem{thm}{Theorem}[section] 
 
\newtheorem{cor}[thm]{Corollary} 
 
\newtheorem{prop}[thm]{Proposition} 
\newtheorem{rmk}[thm]{Remark}

\def\U{\mathscr{U}} 

\def\c{\mathrm{c}} 
\def\d{\mathrm{d}} 

\def\Nset{\mathbb{N}} 
\def\Pset{\mathbb{P}} 
\def\Rset{\mathbb{R}} 
\def\Sset{\mathbb{S}} 
\def\Tset{\mathbb{T}} 
\def\Zset{\mathbb{Z}}

\def\epsilon{\varepsilon} 

\DeclareMathOperator{\esssup}{\mathrm{ess\,sup}} 

\makeatletter 
\@addtoreset{equation}{section} 
\makeatother

\begin{document} 


\title[Kuramoto model with random natural frequencies]%
{Continuum limit of the Kuramoto model with random natural frequencies
 on uniform graphs}
\author{Kazuyuki Yagasaki} 
\address{Department of Applied Mathematics and Physics, Graduate School of Informatics, 
Kyoto University, Yoshida-Honmachi, Sakyo-ku, Kyoto 606-8501, JAPAN} 

\email{yagasaki@amp.i.kyoto-u.ac.jp} 

\date{\today} 

\begin{abstract}
We study the Kuramoto model (KM)
 having random natural frequencies and defined on uniform graphs
 that may be complete, random dense or random sparse.
The natural frequencies are assumed to be independent and identically distributed
 on a bounded interval.
In the previous work,
 the corresponding continuum limit (CL) was proven
 to approximate stable motions in the KM well
 when the natural frequencies are deterministic,
 even if the graph is not uniform,
 although it may not do so for unstable motions and bifurcations.
We show that the method of CLs is still valid
 even when the natural frequencies are random,
 especially uniformly distributed.
In particular, an asymptotically stable family of solutions to the CL
 is proven to behave in the $L^2$ sense
 as if it is an asymptotically stable one in the KM,
 under an appropriate uniform random permutation.
We demonstrate the theoretical results by numerical simulations
 for the KM with uniformly distributed random natural frequencies.
\end{abstract}
\maketitle 


\section{Introduction}

We consider the Kuramoto model (KM) \cite{K75,K84} with natural frequencies
 defined on a graph $G_n=\langle V(G_n),E(G_n),W(G_n)\rangle$:
\begin{equation} 
\frac{\d}{\d t}u_i^n (t) =\omega_i^n 
 +\frac{K}{n \alpha_n} \sum^{n}_{j=1} 
 w_{ij}^n \sin \left( u_j^n(t) - u_i^n(t) \right),\quad i \in [n]:=\{1,2,\ldots,n\}, 
\label{eqn:dsys} 
\end{equation} 
where $u_i^n(t)\in\Sset^1=\Rset/2\pi\Zset$ and $\omega_i^n\in\Rset$,
 respectively, stand for the phase and natural frequency of oscillator at the node $i\in [n]$,
 and $K$ and $\alpha_{n}>0$ are, respectively, the coupling constant and scaling factor.
Here $V(G_n)=[n]$ and $E(G_n)$ are the sets of nodes and edges, respectively, 
 and $W(G_n)$ is an $n\times n$ weight matrix given by 
\begin{equation*} 
(W(G_{n}))_{ij}= 
\begin{cases} 
w_{ij}^n & \mbox{if $(i,j)\in E(G_{n})$};\\ 
0 &\rm{otherwise},
\end{cases} 
\end{equation*}
so that the relation
\[ 
E(G_n)=\{(i,j)\in[n]^2\mid (W(G_{n}))_{ij}\neq 0\}
\]
holds,
 where each edge is represented by an ordered pair of nodes $(i,j)$, 
which is also denoted by $j\to i$, and a loop is allowed.
If $W(G_n)$ is symmetric, 
then $G_n$ represents an undirected weighted graph 
and each edge is also denoted by $i\sim j$ instead of $j\to i$. 
When $G_n$ is a simple graph, 
$W(G_n)$ is a matrix whose elements are $\{0,1\}$-valued. 
When $G_n$ is a random graph, 
$W(G_n)$ is a random matrix. 
We say that $G_n$ is a \emph{dense} graph 
if $\# E(G_n)/(\# V(G_n))^2>0$ (a.s.) as $n \rightarrow\infty$,
 where `$\#$' represents the cardinality of the set (i.e., the number of elements) as usual.
If $\# E(G_n)/(\# V(G_n))^2\rightarrow 0$ as $n \rightarrow \infty$, 
then we call it a \emph{sparse} graph.The scaling factor $\alpha_n$ is one if $G_{n}$ is dense,
 and less than one with $\alpha_{n}\searrow 0$ and $n \alpha_{n} \to\infty$
 as $n\rightarrow \infty$ if $G_{n}$ is sparse.
We are interested in the dynamics of the KM \eqref{eqn:dsys} in the limit of $n\to\infty$,
 and treat random dense and sparse graphs as well as deterministic dense ones at the same time.
So we use the superscript $n$ for the state variables and parameters
 to clarify their dependency on $n$,
 and to introduce the scaling factor $\alpha_n$.
The weight matrix $W(G_n)$ is given as follows.

Let $I=[0,1]$ and let $W\in L^2 (I^2)$ be a nonnegative  function.
If $G_{n}$, $n\in\Nset$, are deterministic dense graphs, then
\begin{equation}
w_{ij}^{n} = \langle W\rangle_{ij}^n
:= n^2 \int_{I_i^n \times I_j^n}W(x,y) \d x\d y,
\label{eq:ddd}
\end{equation}
where
\[
I_i^n:=\begin{cases}
[(i-1)/n,i/n) & \mbox{for $i<n$};\\
[(n-1)/n,1] & \mbox{for $i=n$}.
\end{cases}
\]
If $G_n$, $n\in\Nset$, are random dense graphs, 
then $w_{ij}^{n}=1$ with probability 
\begin{equation}
\mathbb{P}(j \rightarrow i) = \langle W\rangle_{ij}^n, 
\label{eq:ddr}
\end{equation}
where the range of $W$ is contained in $I$. 
If $G_{n}$, $n\in\Nset$, are random sparse graphs, 
then $w_{ij}^{n}=1$ with probability 
\begin{equation} 
\mathbb{P}(j \rightarrow i) = \alpha_{n} \langle \tilde{W}_{n} \rangle_{ij}^n, 
\quad \tilde{W}_{n}(x,y) :=\alpha^{-1}_{n} \wedge W(x,y), 
\label{eq:sdr} 
\end{equation} 
where $\alpha_{n} =n^{-\gamma}$ with $\gamma\in(0,\frac{1}{2})$,
 and $a\wedge b=\min(a,b)$ for $a,b\in\Rset$. 
Here $w_{ij}^n$, $i,j\in[n]$, are also assumed
 to be independently distributed in $j\in[n]$ for each $i\in[n]$
 for random graphs, whether they are dense or sparse.
The function $W(x,y)$ is usually called a \emph{graphon} \cite{L12}.
Such a construction of a random graph
 was given in \cite{M19} and used in \cite{IY23,Y24b,Y24c,Y24d}.

Such coupled oscillators in complex networks
 have recently attracted much attention and have been extensively studied.
They provide many mathematical models in various fields
 such as physics, chemistry, biology, social sciences and engineering.
Among  them, the KM is one of the most representative models
 and has been generalized in several directions.
It has very frequently been subject to research,
 especially as a prototype to discuss the synchronization phenomenon
 in natural sciences and applications.
In particular, the natural frequencies are often assumed
 to be distributed according to some probability density
 since real-world models are considered to consist of a random set of oscillators.
See \cite{S00,PRK01,ABVRS05,ADKMZ08,DB14,PR15,RPJK16}
 for the reviews of vast literature on coupled oscillators in complex networks
 including the KM and its generalizations.

In the previous work \cite{IY23},
 coupled oscillator networks including \eqref{eqn:dsys} were studied
 and shown to be well approximated by the corresponding continuum limits (CLs),
 for instance, which are written as
\begin{equation} 
\frac{\partial}{\partial t}u(t,x)=\omega(x)+K\int _{I}W(x,y)\sin(u(t,y)-u(t,x))\d y,\quad
x\in I,
\label{eqn:csys0}
\end{equation}
for \eqref{eqn:dsys} if the natural frequencies $\omega_i^n$, $i\in[n]$, are deterministic
 and given by $\omega\in L^2(I)$, called a \emph{frequency function}, as
\begin{equation}
\omega_i^n=n\int_{I_i^n}\omega(x)dx, \quad i \in [n].
\label{eqn:omega}
\end{equation}
More general cases in which networks of coupled oscillators are defined 
 on multiple graphs which may be deterministic dense, random dense or random sparse
 were discussed in \cite{IY23}.
Some results on the stability of solutions to the coupled oscillators and CLs
 were also refined in \cite{Y24a}.
Similar results for such networks which are defined on single graphs
 and have the same or equivalently zero natural frequency
 were obtained earlier in \cite{KM17,M14a,M14b,M19}
 although they are not applicable to \eqref{eqn:dsys} and \eqref{eqn:csys0}
 even if the natural frequencies are deterministic.
Such a CL was introduced for the classical KM,
 which depends on a single complete simple graph
 but may have natural frequencies depending on each oscillator,
 without a rigorous mathematical guarantee very early in \cite{E85},
 and fully discussed for the case of equally placed natural frequencies
 very recently in \cite{Y24a}.
Moreover, bifurcations of continuous stationary solutions
 or twisted solutions in CLs for the KM of identical oscillators
 with two-mode interaction depending on two graphs
 or defined on nearest neighbor graphs with feedback control or not were discussed
 by the center manifold reduction technique \cite{HI11},
 which is a standard one in dynamical systems theory \cite{GH83,K04},
 in \cite{Y24b,Y24c,Y24d}.
Similar CLs were also utilized for the KM with nonlocal coupling
 and a single or zero natural frequency
 in \cite{GHM12,M14c,MW17,WSG06}.

On the other hand,
 when the natural frequencies are randomly distributed
 in coupled oscillator networks including the KM \eqref{eqn:dsys},
 a different approach in which an integro-partial differential equation
 called the \emph{Vlasov equations} governing conditional probability density function
 of the phase $u$ depending on time $t$
 under the condition that the natural frequency is $\omega$ are treated
 has more typically been used \cite{ABVRS05,DF18,DB14,RPJK16,S00}
 although its mathematical foundations were provided relatively recently
 in \cite{C15,CN11}.
The Ott-Antonsen ansatz technique \cite{OA08} has also often been utilized
 to analyze the Vlasov equations approximately (see, e.g., \cite{DB14,PR15,RPJK16}).
Only a unimodal or bimodal, typically special, probability density function
 of the natural frequency was treated there.
The approach has also been extended
 to ones defined on single graphs \cite{CM19a,CM19b}.
Compared with those results,
 two advantages of the approach by the use of such CLs as \eqref{eqn:csys0}
 are to deal with the case of deterministic natural frequencies
 and to give more direct description on the dynamics of the coupled oscillators
 without using probability density functions
 although the case of random natural frequencies
 seems difficult to treat directly.

In this paper, we study the KM \eqref{eqn:dsys} on a graph $G_n$
 that is uniform, i.e., $W(x,y)=p$ with a constant $p\in(0,1]$,
 but may be deterministic dense (i.e., complete simple),
 random dense or random sparse,
 when the natural frequencies $\omega_i^n$, $i\in[n]$,
 are randomly distributed on a bounded interval.
We use the CL \eqref{eqn:csys0}, which becomes
\begin{equation} 
\frac{\partial}{\partial t}u(t,x)=\omega(x)+pK\int _{I}\sin(u(t,y)-u(t,x))\d y,
\label{eqn:csys}
\end{equation}
when $W(x,y)=p$,
 to discuss the dynamics of the KM \eqref{eqn:dsys}
 defined on the graphs and having the random natural frequencies.
We show that the CL \eqref{eqn:csys} for the KM \eqref{eqn:dsys}
 with deterministic natural frequencies is useful
 even when the natural frequencies are randomly distributed,
 and describe the dynamics of the KM \eqref{eqn:dsys} with the random natural frequencies,
 based on the results of \cite{IY23,Y24a}.
See Theorem~\ref{thm:3a} below.
In particular, an asymptotically stable family of solutions to the CL \eqref{eqn:csys}
 is proven to behave in the $L^2$ sense
 as if it is an asymptotically stable one  in the KM \eqref{eqn:dsys},
 under an appropriate uniform random permutation.
 
The main results of the previous work \cite{Y24a}
 were on the KM \eqref{eqn:dsys} in which $n$ is odd,
 $\omega_i^n$, $i\in[n]$, are deterministic and given by \eqref{eqn:omega} with
\begin{equation}
\omega(x)=a(x-\tfrac{1}{2}),\quad
a>0,
\label{eqn:lomega}
\end{equation}
so that they are placed equally,
\begin{equation}
\omega_i^n=\frac{a}{2n}(2i-n-1),\quad
i\in[n],
\label{eqn:ep}
\end{equation}
and $G_n$ is complete simple:
The existence of $2^{(n-1)/2}$ one-parameter families of equilibria were detected,
 and their bifurcations and stability were completely determined.
Moreover, it was shown there that
 the CL \eqref{eqn:csys} with \eqref{eqn:lomega}
 has an asymptotically stable family of continuous stationary solutions
 that appears suddenly at the critical value $K_\c=2/\pi$
 when $K$ is increased for the other parameter values fixed,
 while infinitely many unstable families of discontinuous stationary solutions exist there.
Thus, the bifurcation behavior occurring in the CL \eqref{eqn:csys}
 is different from saddle-node bifurcations occurring typically in finite-dimensional dynamical systems
 and from ones in the KM \eqref{eqn:dsys}.

Here we especially discuss the dynamics of the KM \eqref{eqn:dsys}
 with random natural frequencies distributed on a bounded interval uniformly.
In particular, we show that
 the asymptotically stable family of continuous stationary solutions
 to the CL \eqref{eqn:csys} with \eqref{eqn:lomega}
 behaves in the $L^2$ sense
 as if it is an asymptotically stable one  in the KM \eqref{eqn:dsys},
 under an appropriate uniform random permutation
 (see Theorem~\ref{thm:4a} below).
The theoretical results are also demonstrated
 by numerical simulations for the KM \eqref{eqn:dsys}
 having uniformly distributed natural frequencies
 and defined on complete simple, random dense and random sparse graphs.
Synchronized rotational motions
 in feedback control of the KM \eqref{eqn:dsys} with deterministic or random natural frequencies
 defined on such uniform graphs
 and its CL will also be discussed and reported elsewhere \cite{KY25,Y25}.

The outline of this paper is as follows:
In Section~2, we review the previous results of \cite{IY23,Y24a}
 on some relationships between coupled oscillator networks and their CLs
 with deterministic natural frequencies
 in the context of the KM \eqref{eqn:dsys} and CL \eqref{eqn:csys0}
 and on the CL \eqref{eqn:csys} with the frequency function \eqref{eqn:lomega}.
We give our main results on the KM \eqref{eqn:dsys}
 with random natural frequencies distributed on bounded intervals in Section~3,
 and describe their consequences
 to the case of uniformly distributed random natural frequencies in Section~4.
Finally, numerical simulation results for the KM \eqref{eqn:dsys}
 having uniformly distributed natural frequencies
 and defined on uniform deterministic or random graphs
 are given in Section~5.


\section{Previous Results}

\subsection{Relationship between the KM \eqref{eqn:dsys} and CL \eqref{eqn:csys0}}

We first review the results of \cite{IY23,Y24a} on some relationships
 between coupled oscillator networks and their CLs
 when the natural frequencies are deterministic,
 in the context of  \eqref{eqn:dsys} and \eqref{eqn:csys0}.
See Section~2 and Appendices~A and B of \cite{IY23}
 and Section~2 of \cite{Y24a} for more details
 including the proofs of the theorems stated below.
We remark that we can treat the natural frequencies as a graph since
\[
\frac{1}{n}\sum_{j=1}^n\tilde{w}_{ij}^n\tilde{D}(u_j^n(t)-u_i^n(t))=\omega_i^n
\]
when $\tilde{w}_{ij}^n=\omega_i^n$, $j\in[n]$, for each $i\in[n]$,
 and $\tilde{D}(u)=1$.

Henceforth, we assume for the $L^2(I^2)$ function $W(x,y)$
that there exist positive constants $C_j$, $j=1,2$, such that
\begin{equation}
\esssup\displaylimits_{y \in I} \int_I|W (x,y)|dx\leq C_1
\label{eqn:assumpx}
\end{equation}
and
\begin{equation}
\esssup\displaylimits_{x\in I} \int_I|W(x,y)|dy \leq C_2.
\label{eqn:assumpy}
\end{equation}
Here `$\esssup$' represents the essential supremum of the function,
 which equals its supremum except on a set of the Lebesgue measure zero.
If $W(x,y)$ is symmetric,
 then conditions~\eqref{eqn:assumpx} and \eqref{eqn:assumpy} are equivalent.

Let $g(x)\in L^2(I)$
 and let $\mathbf{u}:\Rset\to L^2(I)$ stand for an $L^2(I)$-valued function.
We have the following on the existence and  uniqueness of solutions
 to the initial value problem (IVP) of the CL \eqref{eqn:csys0}
 (see Theorem~2.1 of \cite{IY23}).
 
\begin{thm}
\label{thm:2a}
There exists a unique solution $\mathbf{u}(t)\in C^1(\Rset,L^2(I))$
 to the IVP of \eqref{eqn:csys0} with
\[
u(0,x)=g(x).
\]
Moreover, the solution depends continuously on $g$.
\end{thm}

We next consider the IVP of the KM \eqref{eqn:dsys}
 and turn to the issue on convergence of solutions in \eqref{eqn:dsys}
 to those in the CL \eqref{eqn:csys0}.
Since the right-hand side of \eqref{eqn:dsys} is Lipschitz continuous in $u_i^n$, $i\in[n]$,
 we see by a fundamental result of ordinary differential equations (ODEs)
 (e.g., Theorem~2.1 in Chapter~1 of \cite{CL55})
 that the IVP of \eqref{eqn:dsys} has a unique solution.
Given a solution $u_n(t)=(u_1^n(t),\ldots, u_n^n(t))$ to the IVP of \eqref{eqn:dsys},
 we define an $L^2(I)$-valued function $\mathbf{u}_n:\Rset\to L^2(I)$ as
\begin{equation*}
\mathbf{u}_n(t) = \sum^{n}_{j=1} u_j^n(t) \mathbf{1}_{I_j^n},
\end{equation*}
where $\mathbf{1}_{I_j^n}$ represents the characteristic function of $I_j^n$ for $j\in[n]$.
Let $\|\cdot\|$ denote the norm in $L^2(I)$.
In our setting as stated in Section~1,
 we slightly modify arguments given in the proof of Theorem~2.3 of \cite{IY23}
 to obtain the following.

\begin{thm}
\label{thm:2b}
If $\mathbf{u}_n(t)$ is the solution to the IVP of \eqref{eqn:dsys} such that
\[
\lim_{n\to\infty}\|\mathbf{u}_n(0)-\mathbf{u}(0)\|=0\quad\mbox{a.s.},
\]
then for any $\tau > 0$ we have
\[
\lim_{n \rightarrow \infty}\max_{t\in[0,\tau]}
 \|\mathbf{u}_n(t)-\mathbf{u}(t)\|=0\quad\mbox{a.s.},
\]
where $\mathbf{u}(t)$ represents the solution
 to the IVP of the CL \eqref{eqn:csys0}.
\end{thm}

Note that the statements of Theorem~\ref{thm:2b} and the remainings are valid
 for undirected random graphs 
 since $w_{ij}^n$, $i,j\in[n]$, may only be assumed
 to be independently distributed in $j\in[n]$ for each $i\in[n]$
 in the proof of Theorem~2.3 of \cite{IY23}.

For $\theta\in\Sset^1$,
 let $\boldsymbol{\theta}$ represent the constant function $u=\theta$ in $L^2(I)$.
If $\bar{\mathbf{u}}_n(t)$ is a solution to the KM \eqref{eqn:dsys},
 then so is $\bar{\mathbf{u}}_n(t)+\boldsymbol{\theta}$ for any $\theta\in\Sset^1$.
Similarly, if $\bar{\mathbf{u}}(t)$ is a solution to the CL \eqref{eqn:csys0},
 then so is $\bar{\mathbf{u}}(t)+\boldsymbol{\theta}$ for any $\theta\in\Sset^1$.
Let $\U_n=\{\bar{\mathbf{u}}_n(t)+\boldsymbol{\theta}\mid\theta\in\Sset^1\}$
 and $\U=\{\bar{\mathbf{u}}(t)+\boldsymbol{\theta}\mid\theta\in\Sset^1\}$
 denote the one-parameter families of solutions to \eqref{eqn:dsys} and \eqref{eqn:csys0}, respectively.
We say that $\U_n$ (resp. $\U$) is \emph{stable}
 if solutions starting in its (smaller) neighborhood
 remain in its (larger) neighborhood for $t\ge 0$,
 and \emph{asymptotically stable} if $\U_n$ (resp. $\U$) is stable
 and the distance between such solutions and $\U_n$ (resp. $\U$) in $ L^2(I)$
 converges to zero as $t\to\infty$.
We also obtain the following result from Theorem~2.4 of \cite{Y24a}.

\begin{thm}
\label{thm:2c}
Suppose that the KM\eqref{eqn:dsys} and CL \eqref{eqn:csys0}
 have solutions $\bar{\mathbf{u}}_n(t)$ and $\bar{\mathbf{u}}(t)$, respectively, such that
\begin{equation}
\lim_{n\to\infty}\|\bar{\mathbf{u}}_n(t)-\bar{\mathbf{u}}(t)\|=0\quad\mbox{a.s.}
\label{eqn:thm2c}
\end{equation}
for any $t\in[0,\infty)$.
Then the following hold$:$
\begin{enumerate}
\setlength{\leftskip}{-1.8em}
\item[\rm(i)]
If $\U_n$ is stable $($resp. asymptotically stable$)$ for $n>0$ sufficiently large,
 then $\U$ is also stable $($resp. asymptotically stable$)$.
\item[\rm(ii)]
If $\U$ is stable, then for any $\epsilon,\tau>0$ there exists $\delta>0$ such that
 if $\mathbf{u}_n(t)$ is a solution to the KM \eqref{eqn:dsys} satisfying
\[
\min_{\theta\in\Sset^1}
 \|\mathbf{u}_n(0)-\bar{\mathbf{u}}_n(0)-\boldsymbol\theta\|<\delta\quad\mbox{a.s.},
\]
then
\[
\min_{\theta\in\Sset^1}
 \|\mathbf{u}_n(t)-\bar{\mathbf{u}}_n(t)-\boldsymbol\theta\|<\epsilon\quad\mbox{a.s.}
\]
for any $t\in[0,\tau]$.
Moreover, if $\U$ is asymptotically stable, then
\[
\lim_{t\to\infty}\lim_{n\to\infty}\min_{\theta\in\Sset^1}
 \|\mathbf{u}_n(t)-\bar{\mathbf{u}}_n(t)-\boldsymbol\theta\|=0\quad\mbox{a.s.}
\]
\end{enumerate}
\end{thm}

\begin{rmk}\
\label{rmk:2a}
\begin{enumerate}
\setlength{\leftskip}{-1.8em}
\item[\rm(i)]
In the definition of stability and asymptotic stability of solutions to the CL \eqref{eqn:csys0},
 we cannot distinguish two solutions that are different only in a set with the Lebesgue measure zero.
\item[\rm(ii)]
In Theorem~$2.3$ of {\rm\cite{Y24a}},
 only complete simple graphs were treated
 but Theorem~{\rm\ref{thm:2c}} can be proven similarly
 since its proof relies only on Theorem~$2.2$ of {\rm\cite{Y24a}},
 which was extended to \eqref{eqn:dsys} and \eqref{eqn:csys0}
 in Theorem~{\rm\ref{thm:2b}}.
The same is true for Corollary~$\ref{cor:2a}$
 and Theorems~$\ref{thm:2d}$ and $\ref{thm:2e}$ below.
\end{enumerate}
\end{rmk}

Without assuming the existence of the solutions $\bar{\mathbf{u}}_n(t)$
 to the KM \eqref{eqn:dsys} satisfying \eqref{eqn:thm2c},
 we have the following result as a corollary of  Theorem~\ref{thm:2c}
 (see Corollary~2.6 of \cite{Y24a}).

\begin{cor}
\label{cor:2a}
If $\U$ is stable in the CL \eqref{eqn:csys0},
 then for any $\epsilon,\tau>0$ there exists $\delta>0$ such that
 for $n>0$ sufficiently large, if $\mathbf{u}_n(t)$ is any solution to the KM \eqref{eqn:dsys} satisfying
\[
\min_{\theta\in\Sset^1}\|\mathbf{u}_n(0)-\bar{\mathbf{u}}(0)-\boldsymbol\theta\|<\delta\quad\mbox{a.s.},
\]
then
\[
\min_{\theta\in\Sset^1}
 \|\mathbf{u}_n(t)-\bar{\mathbf{u}}(t)-\boldsymbol\theta\|<\epsilon\quad\mbox{a.s.}
\]
Moreover, if $\U$ is asymptotically stable, then
\[
\lim_{t\to\infty}\lim_{n\to\infty}\min_{\theta\in\Sset^1}
 \|\mathbf{u}_n(t)-\bar{\mathbf{u}}(t)-\boldsymbol\theta\|=0\quad\mbox{a.s.}
\]
\end{cor}

\begin{rmk}
\label{rmk:2b}
Corollary~{\rm\ref{cor:2a}} implies that
 $\U$ behaves as if it is an $($asymptotically$)$ stable family of solutions
 in the KM \eqref{eqn:dsys}, when $n>0$ is sufficiently large.
\end{rmk}

Finally, we have the following results from Theorems~2.7 and 2.9 of \cite{Y24a}.
 
\begin{thm}
\label{thm:2d}
Suppose that the hypothesis of Theorem~$\ref{thm:2c}$ holds.
Then the following hold$:$
\begin{enumerate}
\setlength{\leftskip}{-1.8em}
\item[\rm(i)]
If $\U_n$ is unstable for $n>0$ sufficiently large
 and no stable solution to the KM \eqref{eqn:dsys} converges to $\U$ a.s. as $n\to\infty$,
 then $\U$ is unstable$;$
\item[\rm(ii)]
If $\U$ is unstable, then so is $\U_n$ for $n>0$ sufficiently large.
\end{enumerate}
\end{thm}

\begin{thm}
\label{thm:2e}
If $\U$ is unstable,
 then for any $\epsilon,\delta>0$ there exists $\tau>0$ such that for $n>0$ sufficiently large
\[
\min_{\theta\in\Sset^1}\|\mathbf{u}_n(\tau)-\bar{\mathbf{u}}(\tau)-\boldsymbol{\theta}\|
 >\epsilon\quad\mbox{a.s.},
\]
where $\mathbf{u}_n(t)$ is a solution to the KM \eqref{eqn:dsys} satisfying
\[
\min_{\theta\in\Sset^1}\|\mathbf{u}_n(0)-\bar{\mathbf{u}}(0)-\boldsymbol{\theta}\|<\delta\quad\mbox{a.s.}
\]
\end{thm}

\begin{rmk}\
\label{rmk:2c}
\begin{enumerate}
\setlength{\leftskip}{-1.8em}
\item[\rm(i)]
Only under the hypothesis of Theorem~{\rm\ref{thm:2d}},
 $\U$ is not necessarily unstable even if $\U_n$ is unstable.
Moreover, $\U$ may be asymptotically stable
 even if $\U_n$ is unstable for $n>0$ sufficiently large.
See {\rm\cite{Y24a}} for such an example.
\item[\rm(ii)]
Theorem~{\rm\ref{thm:2e}} implies that
 $\U$ behaves as if it is an unstable family of solutions
 in the KM \eqref{eqn:dsys}, when $n>0$ is sufficiently large.
\end{enumerate}
\end{rmk}

Thus, the relationship
 between the KM \eqref{eqn:dsys} and CL \eqref{eqn:csys0} is subtle.

\subsection{CL \eqref{eqn:csys} with the frequency function \eqref{eqn:lomega}}

We next outline the results of \cite{Y24a} on the CL \eqref{eqn:csys}
 with the frequency function \eqref{eqn:lomega}.
Here we assume that $a>0$.
A different treatment is needed for $a=0$
 (see, e.g., Proposition~3.2 of \cite{Y24b}).

As shown in \cite{IY23},
 the CL \eqref{eqn:csys} generally has the stationary solutions
\begin{equation}
u(t,x)=U(x)+\Omega t+\theta,\quad
U(x)=\arcsin\left(\frac{\omega(x)-\Omega}{pKC}\right),
\label{eqn:csol}
\end{equation}
where $\theta\in\Sset^1$ is an arbitrary constant,
 the range of the function $\arcsin$ is $[-\tfrac{1}{2}\pi,\tfrac{1}{2}\pi]$ and
\begin{equation}
\Omega=\int_I \omega(x)\d x,
\label{eqn:Omega}
\end{equation}
if there exists a constant $C>0$ such that
\begin{equation}
C=\int_I \sqrt{1-\left(\frac{\omega(x)-\Omega}{pKC}\right)^2}\d x.
\label{eqn:CU}
\end{equation}
A solution of the form \eqref{eqn:csol} to the CL \eqref{eqn:csys}
 was also obtained in \cite{E85}
 although the constant term 
 was not contained.

\begin{figure}
\includegraphics[scale=0.5]{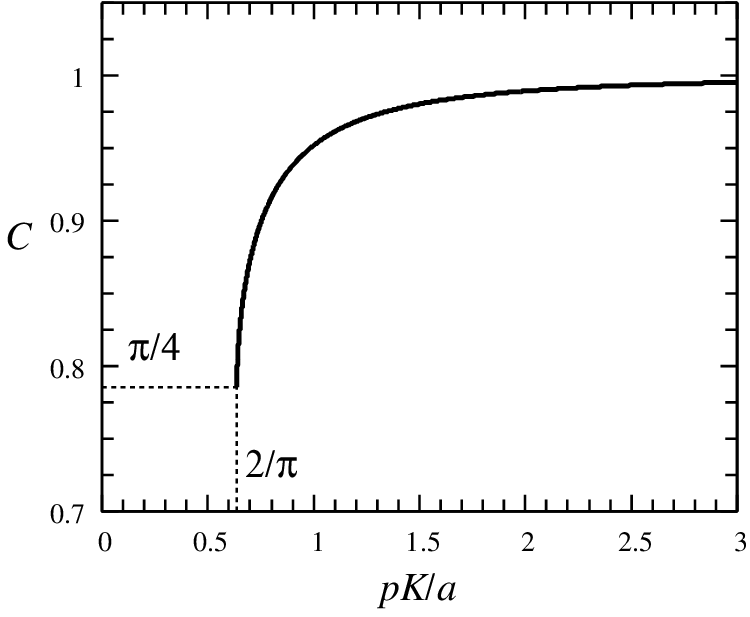}
\caption{Dependence of $C$ on $pK/a$ in \eqref{eqn:Cex}.
\label{fig:2a}}
\end{figure}

For the frequency function \eqref{eqn:lomega},
 the family \eqref{eqn:csol} of continuous stationary solutions becomes
\begin{equation}
u(t,x)=U(x)+\theta,\quad
U(x)=\arcsin\left(\frac{a(x-\tfrac{1}{2})}{pKC}\right),
\label{eqn:csol0}
\end{equation}
where $\theta\in\Sset^1$ is an arbitrary constant
 and the constant $C$ satisfies \eqref{eqn:CU} with \eqref{eqn:lomega}:
\begin{equation}
C=\frac{pKC}{a}\left(\arcsin\left(\frac{a}{2pKC}\right)
 +\frac{a}{2pKC}\sqrt{1-\left(\frac{a}{2pKC}\right)^2}\right),
\label{eqn:Cex}
\end{equation}
Note that $\Omega=0$ for \eqref{eqn:lomega}.
The dependence of $C$ on $pK/a$
 which is calculated from \eqref{eqn:Cex} is displayed in Fig.~\ref{fig:2a}.
We easily see that the CL \eqref{eqn:csys}
 has another family of continuous stationary solutions
\begin{equation}
u(t,x)=\pi-U(x)+\theta,
\label{eqn:csol1}
\end{equation}
where $U(x)$ is the same as in \eqref{eqn:csol0}.
Obviously, the two families \eqref{eqn:csol0} and \eqref{eqn:csol1} never coincide.
We have the following result (see Theorem~7.1 of \cite{Y24a}).

\begin{thm}
\label{thm:2f}
The two families \eqref{eqn:csol0} and \eqref{eqn:csol1} of continuous stationary solutions
 exist in the CL \eqref{eqn:csys} with the frequency function \eqref{eqn:lomega}
 if and only if $pK/a\ge 2/\pi$.
\end{thm}

Theorem~\ref{thm:2f} does not deny
 the existence of other discontinuous stationary solutions to the CL \eqref{eqn:csys}.
Indeed, Eq.~\eqref{eqn:csys} has infinitely many one-parameter families
 of discontinuous stationary solutions of the form
\begin{equation}
u(t,x)=\begin{cases}
U(x)+\theta
 & \mbox{for $x\in[0,1]\setminus\hat{I}$;}\\
\pi-U(x)+\theta & \mbox{for $x\in\hat{I}_j^+$, $j\in[m_+]$;}\\
-U(x)-\pi+\theta & \mbox{for $x\in\hat{I}_j^-$, $j\in[m_-]$,}
\end{cases}
\label{eqn:csol2}
\end{equation}
where $\theta\in\Sset^1$ is an arbitrary constant,
 $m_\pm$ are nonnegative integers and may be infinite,
 $\hat{I}_j^{\pm}\subset[0,1]$, $j\in[m_\pm]$, are intervals
 such that $\hat{I}_j^-\subset[0,\tfrac{1}{2}]$, $\hat{I}_j^+\subset[\tfrac{1}{2},1]$,
 $\hat{I}_j^-\cap \hat{I}_k^-,\hat{I}_j^+\cap \hat{I}_k^+=\emptyset$ and
\[
\hat{I}=\bigcup_{j=1}^{m_-}\hat{I}_j^-\cup\bigcup_{j=1}^{m_+}\hat{I}_j^+,
\]
and the constant $C$ in $U(x)$ satisfies
\[
C=\int_{[0,1]\setminus\hat{I}}\sqrt{1-\left(\frac{a(x-\tfrac{1}{2})}{pKC}\right)^2}\d x
 -\int_{\hat{I}}\sqrt{1-\left(\frac{a(x-\tfrac{1}{2})}{pKC}\right)^2}\d x.
\]
Here the interiors of $\hat{I}_j^\pm$, $j\in[m_\pm]$, may be empty.
See Section~7 of \cite{Y24a} for the details,
 although the formula for $C$ there can reduce to  the above with $p=1$
We have the following result from Theorem~7.2 of \cite{Y24a}.

\begin{thm}\
\label{thm:2g}
\begin{enumerate}
\setlength{\leftskip}{-1.8em}
\item[(i)]
The family of continuous stationary solutions given by \eqref{eqn:csol0}
 is asymptotically stable,
 while the family of continuous stationary solutions given by \eqref{eqn:csol1}
 is unstable.
\item[(ii)]
The family of discontinuous stationary solutions given by \eqref{eqn:csol2} is unstable
 if it is different from \eqref{eqn:csol0} in the sense of $L^2(I)$.
\end{enumerate}
\end{thm}
 
Theorem~\ref{thm:2d}(i) plays an important role
 in the proof of Theorem~\ref{thm:2g}(ii) (see Section~7 of \cite{Y24a}).


\section{Main Results}

We turn to the KM \eqref{eqn:dsys} with random natural frequencies $\omega_i^n$, $i\in[n]$.
We assume that 
 they are independent and identically distributed on a bounded closed interval,
 and without loss of generality that
 the interval is written as $[-\tfrac{1}{2}a,\tfrac{1}{2}a]$ for some $a>0$,
 taking a rotational frame rotating with the mean of the natural frequencies
 if necessary.

Let $\rho(\omega)$ be the probability density function of $\omega_i^n$, $i\in[n]$,
 and assume that $\rho(\omega)>0$ on $[-\tfrac{1}{2}a,\tfrac{1}{2}a]$.
Let $\xi_n:[n]\to[n]$ denote the random permutation of $[n]$,
 such that
\[
\omega_{\xi_n(1)}^n<\ldots<\omega_{\xi_n(n)}^n\quad\mbox{a.s.}
\]
We easily see that $\xi_n:[n]\to[n]$ is also uniform
 since $\omega_i^n$, $i\in[n]$, are independent and identically distributed.
Let $F(x)$ denote their probability distribution function, i.e.,
\[
F(\omega)=\int_{-a/2}^{\omega}\rho(z)\d z,\quad
\omega\in[-\tfrac{1}{2}a,\tfrac{1}{2}a],
\]
and let
\[
\nu_n(i)=n\int_{I_i^n}F^{-1}(x)\d x,\quad
i\in[n].
\]
Note that the inverse function $F^{-1}(x)$ always exists
 since $F(x)$ is strictly increasing.
It is clear that both $F(\omega)$ and $F^{-1}(x)$ are continuous.

\begin{prop}
\label{prop:3a}
The random variable $\max_{i\in[n]}|\omega_{\xi_n(i)}^n-\nu_n(i)|$
 converges to zero a.s. as $n\to\infty$.
\end{prop}

\begin{proof}
Let 
\[
F_n(\omega)=\frac{1}{n}\sum_{i=1}^n \mathbf{1}_{\{\omega_i^n<\omega\}},
\]
where $\mathbf{1}_{\{\omega_i^n<\omega\}}$ represents the characteristic function of the set
 $\{\omega_i^n<\omega\}$ for $i\in[n]$.
By the strong law of large numbers for empirical distributions
 (see, e.g., Theorem~1.1.3 in Part~III of \cite{S05}), we have
\begin{equation}
\lim_{n\to\infty}\sup_{\omega\in[-a/2,a/2]}|F_n(\omega)-F(\omega)|=0\quad\mbox{a.s.}
\label{eqn:prop3a}
\end{equation}
By the continuity of $F(\omega)$, we compute
\[
\int_{I_i^n}F^{-1}(x)\d x=\frac{1}{n}F^{-1}\left(\frac{i-1}{n}\right)+O(n^{-2}),
\]
so that
\[
F(\nu_n(i))=F\left(n\int_{I_i^n}F^{-1}(x)\d x\right)=\frac{i-1}{n}+O(n^{-1}).
\]
Since
\[
F_n(\omega_{\xi_n(i)}^n)=\frac{i-1}{n}\quad\mbox{a.s.},
\]
we obtain
\[
F(\nu_n(i))=F_n(\omega_{\xi_n(i)}^n)+O(n^{-1})\quad\mbox{a.s.}
\]
Hence, we have
\begin{align*}
&
\max_{i\in[n]}|F(\nu_n(i))-F(\omega_{\xi_n(i)}^n)|\\
&
\le\max_{i\in[n]}|F(\nu_n(i))-F_n(\omega_{\xi_n(i)}^n)|
 +\max_{i\in[n]}|F_n(\omega_{\xi_n(i)}^n)-F(\omega_{\xi_n(i)}^n)|\\
&
\le\max_{i\in[n]}|F(\nu_n(i))-F_n(\omega_{\xi_n(i)}^n)|
 +\sup_{\omega\in[-a/2,a/2]}|F_n(\omega)-F(\omega)|,
\end{align*}
which yields
\[
\lim_{n\to\infty}\max_{i\in[n]}|F(\nu_n(i))-F(\omega_{\xi_n(i)}^n)|=0\quad\mbox{a.s.}
\]
by \eqref{eqn:prop3a} when taking the limit $n\to\infty$.
Since $F^{-1}(x)$ is continuous, we obtain the desired result.
\end{proof}

Define the random operator $T_{\xi_n}:\Tset^n=\prod_{i=1}^n\Sset^1\to\Tset^n$
 for the uniform random permutation $\xi_n:[n]\to[n]$ as
\[
T_{\xi_n}(u_1^n(t),\ldots,u_n^n(t))
 =(u_{\xi_n(1)}^n(t),\ldots,u_{\xi_n}^n(t)),
\]
and write
\[
T_{\xi_n}\mathbf{u}_n(t)=\sum_{i=1}^nu_{\xi_n(i)}^n(t)\mathbf{1}_i^n.
\]
Let  $\bar{\mathbf{u}}(t)$ be a solutions to the CL \eqref{eqn:csys},
 and let $\U=\{\bar{\mathbf{u}}(t)+\boldsymbol{\theta}\mid\theta\in\Sset^1\}$,
 as in Section~2.
Using Theorems~\ref{thm:2b} and \ref{thm:2e}, Corollary~\ref{cor:2a}
 and Proposition~\ref{prop:3a},
 we prove the following theorem.

\begin{thm}
\label{thm:3a}
Suppose that $\omega_i^n$, $i\in[n]$, are independent and identically distributed
 on the finite interval $[-\tfrac{1}{2}a.\tfrac{1}{2}a]$ with the probability distribution function $F(\omega)$.
Let $\omega(x)=F^{-1}(x)$.
Then the following hold$:$
\begin{enumerate}
\setlength{\leftskip}{-1.6em}
\item[\rm(i)]
If $\mathbf{u}_n(t)$ is the solution to the IVP of \eqref{eqn:dsys}
 with the initial condition
\[
\lim_{n\to\infty}\|T_{\xi_n}\mathbf{u}_n(0)-\mathbf{u}(0)\|=0\quad\mbox{a.s.},
\]
then for any $\tau > 0$ we have
\[
\lim_{n \rightarrow \infty}
\max_{t\in[0,\tau]}\|T_{\xi_n}\mathbf{u}_n(t)-\mathbf{u}(t)\|=0\quad\mbox{a.s.},
\]
where $\mathbf{u}(t)$ represents the solution
 to the IVP of the CL \eqref{eqn:csys}.
\item[\rm(ii)]
If $\U$ is stable, then for any $\epsilon,\tau>0$ there exists $\delta>0$
 such that for any $n>0$ sufficiently large,
 if $\mathbf{u}_n(t)$ is any solution to the KM \eqref{eqn:dsys} satisfying
\[
\min_{\theta\in\Sset^1}
 \|T_{\xi_n}\mathbf{u}_n(0)-\bar{\mathbf{u}}(0)-\boldsymbol\theta\|<\delta\quad\mbox{a.s.},
 \]
then
\[
\max_{t\in[0,\tau]}\min_{\theta\in\Sset^1}
 \|T_{\xi_n}\mathbf{u}_n(t)-\bar{\mathbf{u}}(t)-\boldsymbol\theta\|<\epsilon\quad\mbox{a.s.}
\]
Moreover, if $\U$ is asymptotically stable, then
\[
\lim_{t\to\infty}\lim_{n\to\infty}\min_{\theta\in\Sset^1}
 \|T_{\xi_n}\mathbf{u}_n(t)-\bar{\mathbf{u}}(t)-\boldsymbol\theta\|=0\quad\mbox{a.s.}
\]
\item[\rm(iii)]
If $\U$ is unstable,
 then for any $\epsilon,\delta>0$ there exists $\tau>0$ such that for $n>0$ sufficiently large
\[
\min_{\theta\in\Sset^1}\|T_{\xi_n}\mathbf{u}_n(\tau)-\bar{\mathbf{u}}(\tau)-\boldsymbol{\theta}\|
 >\epsilon\quad\mbox{a.s.},
\]
where $\mathbf{u}_n(t)$ is a solution to the KM \eqref{eqn:dsys} satisfying
\[
\min_{\theta\in\Sset^1}\|T_{\xi_n}\mathbf{u}_n(0)-\bar{\mathbf{u}}(0)-\boldsymbol\theta\|<\delta\quad\mbox{a.s.}
\]
\end{enumerate}
\end{thm}

\begin{proof}
We additionally consider the two KMs
\begin{align}
&
\frac{\d}{\d t}u_{\xi_n(i)}^n (t)\notag\\
&
=\omega_{\xi_n(i)}^n
 +\frac{K}{n \alpha_n} \sum^{n}_{j=1}
 w_{\xi_n(i)\xi_n(j)}^n \sin \left(u_{\xi_n(j)}^n(t)-u_{\xi_n(i)}^n(t)\right),\quad
i\in [n],
\label{eqn:dsys2}
\end{align}
and
\begin{align}
\frac{\d}{\d t}v_i^n(t)=\nu_n(i)
 +\frac{K}{n \alpha_n} \sum^{n}_{j=1}
 w_{\xi_n(i)\xi_n(j)}^n \sin \left(v_j^n(t)-v_i^n(t)\right),\quad
i\in [n].
\label{eqn:dsys3}
\end{align}
By Proposition~\ref{prop:3a},
 the frequency $\omega_{\xi_n(i)}^n$ in \eqref{eqn:dsys2}
 converges to $\nu_n(i)$ in \eqref{eqn:dsys3} a.s. as $n\to\infty$ for each $i\in[n]$.

The CL corresponding to the KM \eqref{eqn:dsys3} is also given
 by \eqref{eqn:csys} with $u=v$
 since $\nu_n(i)$, $i\in[n]$, satisfy \eqref{eqn:omega}
 with $\omega(x)=F^{-1}(x)$.
Hence, the statements of Theorems~\ref{thm:2b} and \ref{thm:2e}
 and Corollary~\ref{cor:2a} hold
 for the KM \eqref{eqn:dsys3} and CL \eqref{eqn:csys},
 so that parts~(i)-(iii) hold when replacing $\mathbf{u}_n$ with $\mathbf{v}_n$.
We easily see that there exists $\tilde{\omega}_n(x)\in L^2(I)$
 such that $\omega_{\xi_n(i)}^n$, $i\in[n]$, satisfy \eqref{eqn:omega}
 when replacing $\omega(x)$ with $\tilde{\omega}_n(x)$, and
\[
\lim_{n\to\infty}\|F^{-1}(x)-\tilde{\omega}_n(x)\|=0\quad\mbox{a.s.}
\]
Hence, for any $\epsilon,\tau>0$ there exists $N>0$ such that if $n>N$ and
\[
\|T_{\xi_n}\mathbf{u}_n(0)-\mathbf{v}_n(0)\|=0\quad\mbox{a.s.},
\]
then
\[
\|T_{\xi_n}\mathbf{u}_n(\tau)-\mathbf{v}_n(\tau)\|<\epsilon
\quad\mbox{a.s.}
\]
by the continuity of solutions with respect to parameters
 (e.g., Theorem~7.4 in Chapter~1 of \cite{CL55}).
Thus, we obtain the desired results.
\end{proof}

\begin{rmk}\
\label{rmk:3b}
\begin{enumerate}
\setlength{\leftskip}{-1.8em}
\item[\rm(i)]
Obviously, the statements of Theorem~$\ref{thm:3a}$ would not hold
 when the graphs $G_n$, $n\in\Nset$, were not uniform.
\item[\rm(ii)]
Theorem~{\rm\ref{thm:3a}(ii)} implies that
 the asymptotically stable family $\U$ in the CL \eqref{eqn:csys}
 behaves in the $L^2(I)$ sense 
 as if it is an asymptotically stable family of solutions to the KM \eqref{eqn:dsys}
 under an appropriate uniform random permutation.
\end{enumerate}
\end{rmk}


\section{Uniform Distribution Case}

We now consider the case in which $\omega_i^n$, $i\in[n]$,
 are uniformly distributed on $[-\tfrac{1}{2}a,\tfrac{1}{2}a]$, i.e., $\rho(\omega)=1/a$.
Using the results of Section~2.2,
 we can obtain stronger statements for the KM \eqref{eqn:dsys} than Theorem~\ref{thm:3a}.

We compute
\[
F(\omega)=\frac{1}{a}(x+\tfrac{1}{2}a)=\frac{2x+a}{2a}
\]
and
\[
F^{-1}(x)=a(x-\tfrac{1}{2}).
\]
Moreover, Eq.~\eqref{eqn:ep} holds when replacing $\omega_i^n$ with $\nu_n(i)$ for $i\in[n]$,
 i.e., they are equally placed:
\[
\nu_n(i)=\frac{a}{2n}(2i-n-1).
\]
Let $\U=\{\bar{\mathbf{u}}(t)+\boldsymbol{\theta}\mid\theta\in\Sset^1\}$
 (resp. $\hat{\U}=\{\hat{\mathbf{u}}(t)+\boldsymbol{\theta}\mid\theta\in\Sset^1\}$)
 denote the family of solutions given by \eqref{eqn:csol0}
 (resp. \eqref{eqn:csol1} or \eqref{eqn:csol2}) to the CL \eqref{eqn:csys}
 with \eqref{eqn:lomega}.
By Theorem~\ref{thm:2g}, $\U$ is asymptotically stable
 and $\hat{\U}$ is unstable in the CL \eqref{eqn:csys}.
Using Theorem~\ref{thm:3a}(ii) and (iii),
 we obtain the following.

\begin{thm}
\label{thm:4a}
Suppose that $\omega_i^n$, $i\in[n]$, 
 are uniformly distributed on $[-\tfrac{1}{2}a,\tfrac{1}{2}a]$.
Then the following hold$:$
\begin{enumerate}
\setlength{\leftskip}{-1.6em}
\item[\rm(i)]
For any $\epsilon,\tau>0$ there exists $\delta>0$ such that for $n>0$ sufficiently large,
 if $\mathbf{u}_n(t)$ is a solution to the KM \eqref{eqn:dsys} satisfying
\[
\min_{\theta\in\Sset^1}
 \|T_{\xi_n}\mathbf{u}_n(0)-\bar{\mathbf{u}}(0)-\boldsymbol\theta\|<\delta
\quad\mbox{a.s.},
 \]
then
\[
\max_{t\in[0,\tau]}\min_{\theta\in\Sset^1}
 \|T_{\xi_n}\mathbf{u}_n(t)-\bar{\mathbf{u}}(t)-\boldsymbol\theta\|<\epsilon
\quad\mbox{a.s.}
\]
Moreover,
\[
\lim_{t\to\infty}\lim_{n\to\infty}\min_{\theta\in\Sset^1}
 \|T_{\xi_n}\mathbf{u}_n(t)-\bar{\mathbf{u}}(t)-\boldsymbol\theta\|
 =0\quad\mbox{a.s.}
\]
\item[\rm(ii)]
For any $\epsilon,\delta>0$ there exists $\tau>0$ such that for $n>0$ sufficiently large
\[
\min_{\theta\in\Sset^1}\|T_{\xi_n}\mathbf{u}_n(\tau)-\hat{\mathbf{u}}(\tau)-\boldsymbol{\theta}\|
 >\epsilon\quad\mbox{a.s.},
\]
where $\mathbf{u}_n(t)$ is a solution to the KM \eqref{eqn:dsys} satisfying
\[
\min_{\theta\in\Sset^1}\|T_{\xi_n}\mathbf{u}_n(0)-\hat{\mathbf{u}}(0)-\boldsymbol\theta\|<\delta\quad\mbox{a.s.}
\]
\end{enumerate}
\end{thm}

\begin{proof}
Parts~(i) and (ii) are special cases of Theorem~\ref{thm:3a}(ii) and (iii) in which $F(\omega)=1/a$.
We note that by Theorem~\ref{thm:2g}
 $\U$ and $\hat{\U}$ are asymptotic stable and unstable families of stationary solutions
 in the CL \eqref{eqn:csys}, to obtain the desired result.
\end{proof}


The statements of Theorem~\ref{thm:4a} also hold
 without the uniform random permutation
 when $\omega_i^n$, $i\in[n]$, are deterministic
 and $\lim_{n\to\infty}\max_{i\in[n]}|\omega_i^n-\nu_n(i)|=0$,
 in particular they are equally placed, i.e., $\omega_i^n=\nu_n(i)$, $i\in[n]$,
 as follows:
 
\begin{cor}
\label{cor:4a}
Suppose that $\omega_i^n$, $i\in[n]$, are deterministic
 and $\lim_{n\to\infty}\max_{i\in[n]}|\omega_i^n-\nu_n(i)|=0$.
Then the following hold$:$
\begin{enumerate}
\setlength{\leftskip}{-1.6em}
\item[\rm(i)]
For any $\epsilon,\tau>0$ there exists $\delta>0$ such that for $n>0$ sufficiently large, if
\[
\min_{\theta\in\Sset^1}
 \|\mathbf{u}_n(0)-\bar{\mathbf{u}}(0)-\boldsymbol\theta\|<\delta\quad\mbox{a.s.},
 \]
then
\[
\max_{t\in[0,\tau]}\min_{\theta\in\Sset^1}
 \|\mathbf{u}_n(t)-\bar{\mathbf{u}}(t)-\boldsymbol\theta\|<\epsilon
\quad\mbox{a.s.}
\]
Moreover,
\[
\lim_{t\to\infty}\lim_{n\to\infty}\min_{\theta\in\Sset^1}
 \|\mathbf{u}_n(t)-\bar{\mathbf{u}}(t)-\boldsymbol\theta\|
 =0\quad\mbox{a.s.}
\]
\item[\rm(ii)]
For any $\epsilon,\delta>0$ there exists $\tau>0$ such that for $n>0$ sufficiently large
\[
\min_{\theta\in\Sset^1}\|\mathbf{u}_n(\tau)-\hat{\mathbf{u}}(\tau)-\boldsymbol{\theta}\|
 >\epsilon\quad\mbox{a.s.},
\]
where $\mathbf{u}_n(t)$ is a solution to the KM \eqref{eqn:dsys} satisfying
\[
\min_{\theta\in\Sset^1}\|\mathbf{u}_n(0)-\hat{\mathbf{u}}(0)-\boldsymbol\theta\|<\delta\quad\mbox{a.s.}
\]
\end{enumerate}
\end{cor}

In Section~3 of  \cite{IY23},
 numerical simulation results for the KM \eqref{eqn:dsys}
 having equally placed natural frequencies
 and defined on complete simple and uniform random dense and sparse graphs
 were given and its responses were observed to converge
 to the solutions corresponding to \eqref{eqn:dsys} of the CL \eqref{eqn:csys}
 as $t\to\infty$ in the $L^2(I)$ sense
 although fluctuations due to the randomness were present
 for the random graphs.


\section{Numerical Simulations}

Finally, we give numerical simulation results for the KM \eqref{eqn:dsys}
 when the natural frequencies $\omega_i^n$, $i\in[n]$,
 are uniformly distributed on $[-\tfrac{1}{2}a,\tfrac{1}{2}a]$, as in Section~4.
We choose this simple case
 since the existence and stability of continuous and discontinuous solutions
 in the corresponding CL \eqref{eqn:csys} with \eqref{eqn:lomega} are detected
 as described in Section~2.2,
 and the dynamics of the KM \eqref{eqn:dsys} with the random natural frequencies
 are well understood, as stated in Theorem~\ref{thm:4a}.
The situation is different for non-uniform distributions of the natural frequencies.

\begin{figure}[t]
\begin{minipage}[t]{0.495\textwidth}
\begin{center}
\includegraphics[scale=0.32]{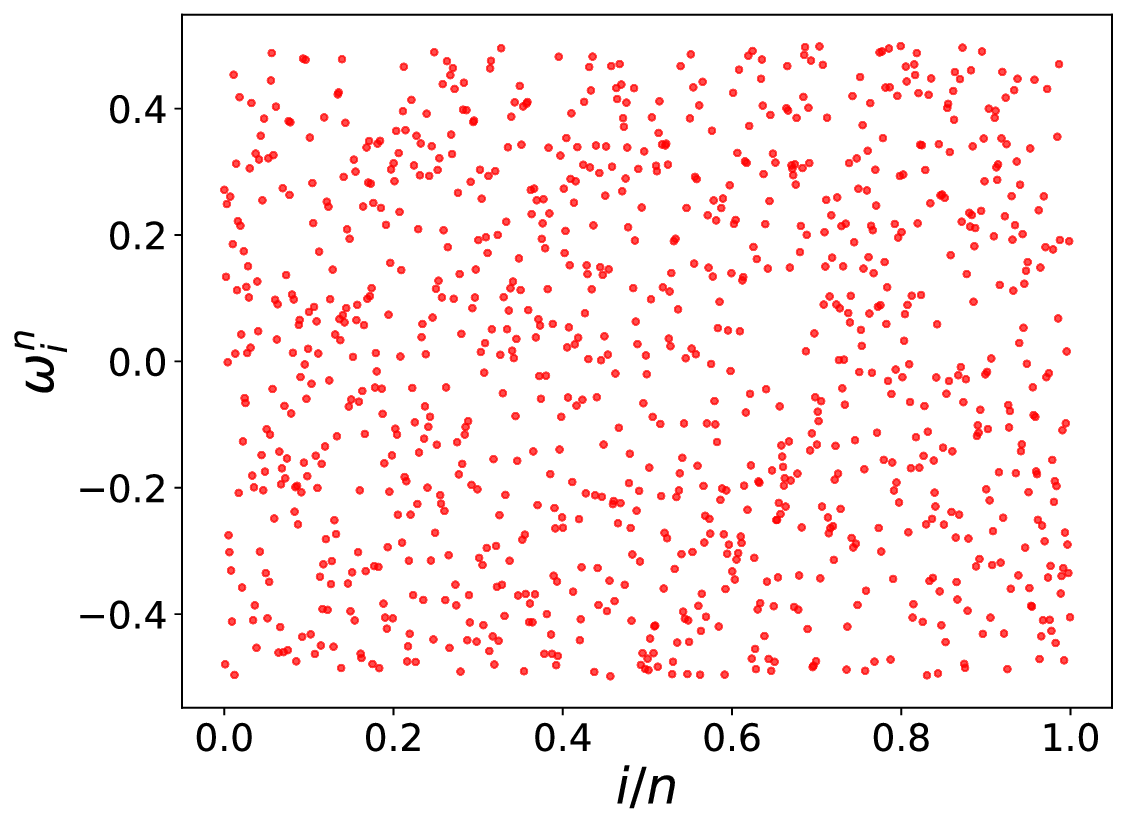}\\
{\footnotesize(a)}
\end{center}
\end{minipage}
\begin{minipage}[t]{0.495\textwidth}
\begin{center}
\includegraphics[scale=0.32]{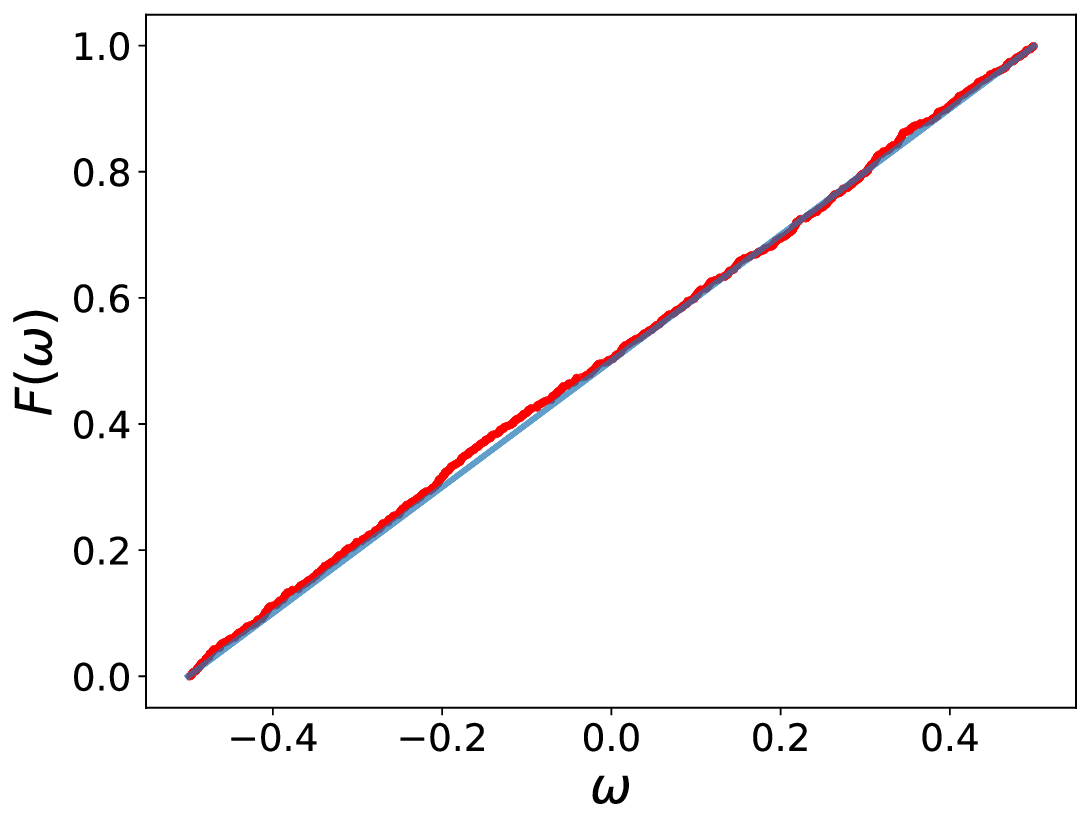}\\
{\footnotesize(b)}
\end{center}
\end{minipage}
\caption{Random natural frequencies $\omega_j^n$, $i\in[n]$,
 uniformly distributed on $[-\tfrac{1}{2}a,\tfrac{1}{2}a]$ with $n=1000$ and $a=1$:
(a) Numerically computed sample;
(b) distribution function $F(\omega)$ obtained from the sample of Fig.~(a).
Numerically generated ones are plotted as small red disks in both figures,
 and the uniform distribution as a blue line for comparison in Fig.~(b).
\label{fig:5a1}}
\end{figure}

\begin{figure}[t]
\begin{minipage}[t]{0.495\textwidth}
\begin{center}
\includegraphics[scale=0.3]{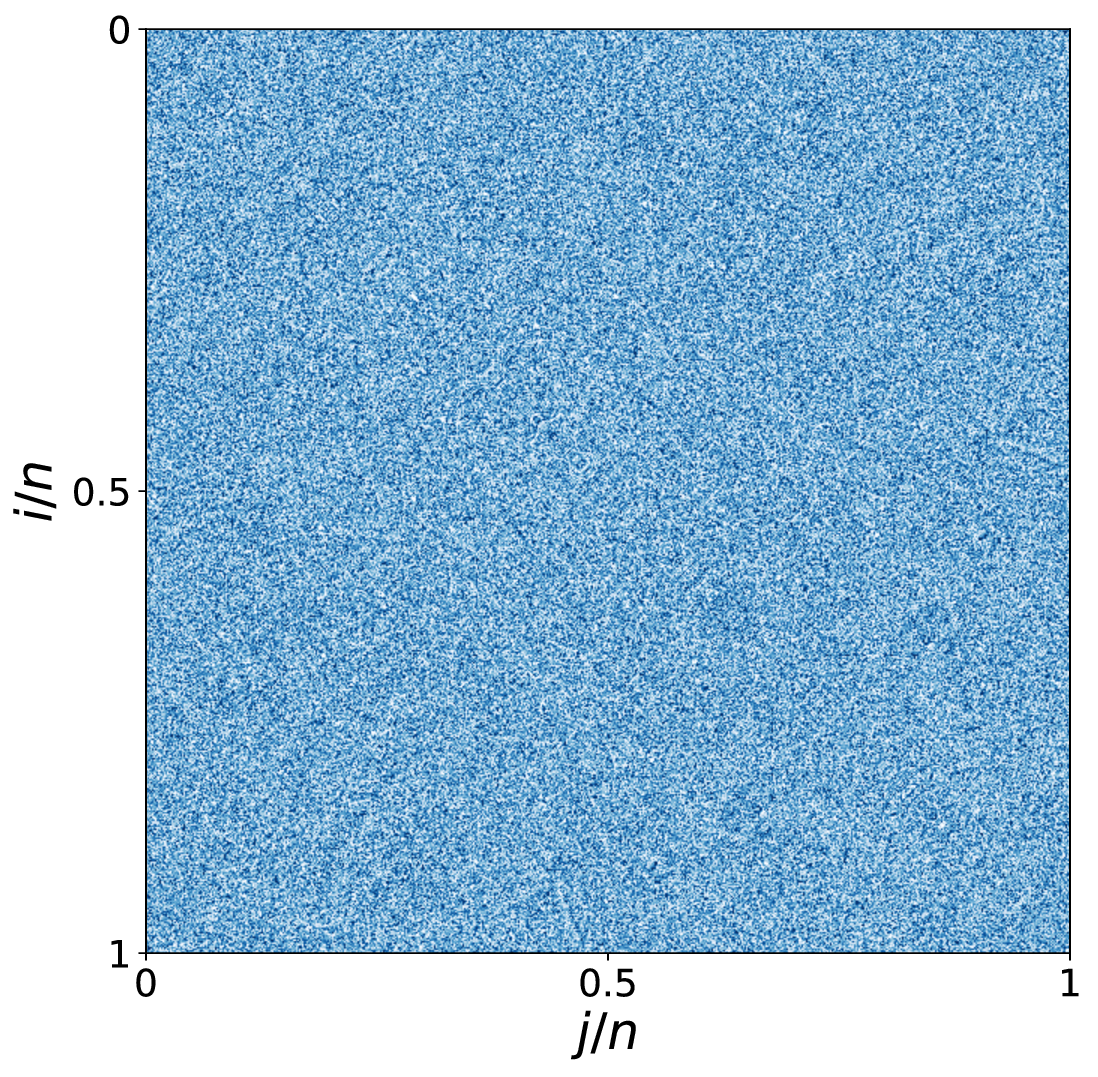}\\
{\footnotesize(a)}
\end{center}
\end{minipage}
\begin{minipage}[t]{0.495\textwidth}
\begin{center}
\includegraphics[scale=0.3]{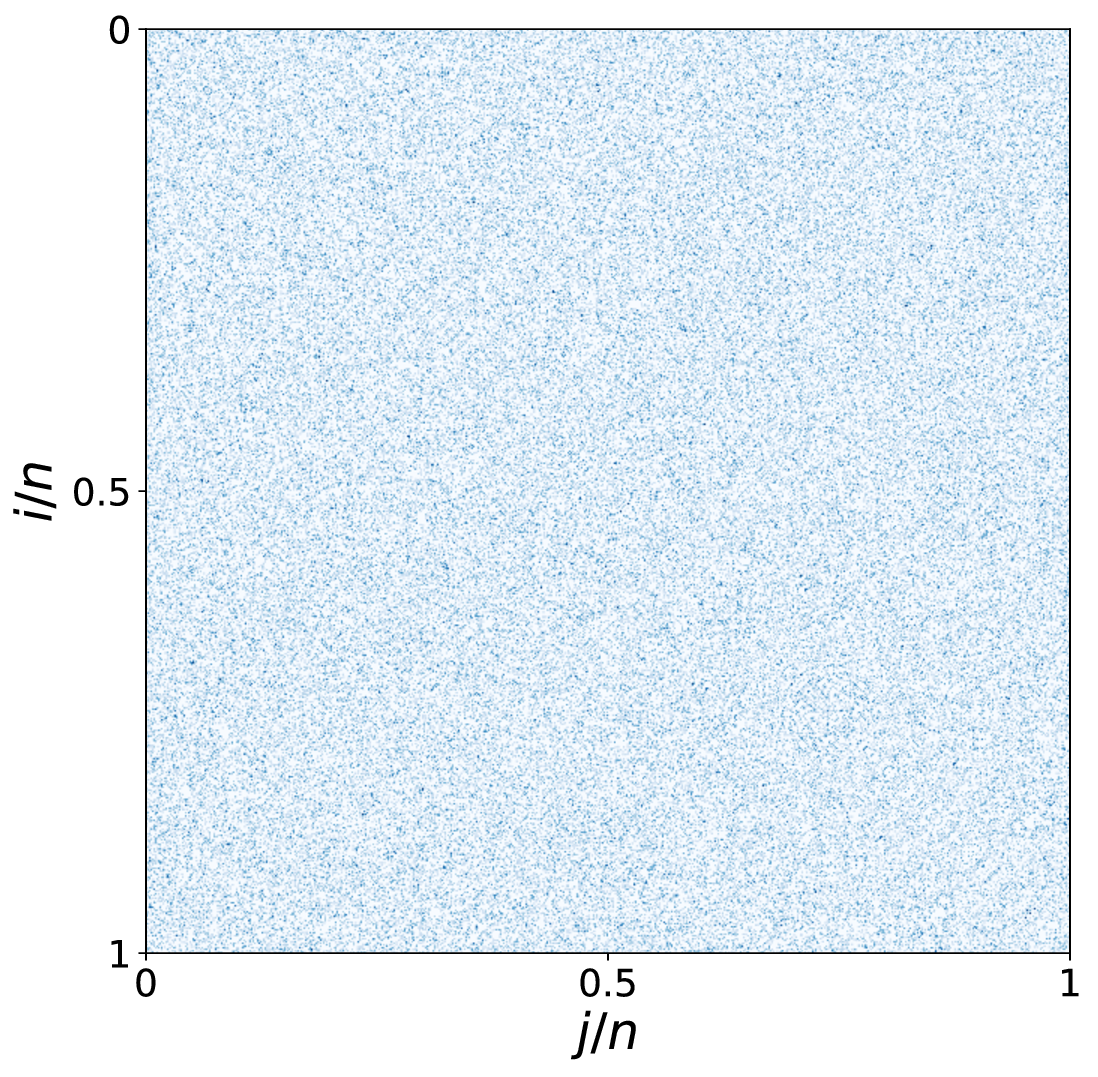}\\
{\footnotesize(b)}
\end{center}
\end{minipage}
\caption{Pixel picture of sampled weight matrices of the random undirected graphs
 given by  $w_{ij}^n=1$, $i,j\in[n]$, with probability \eqref{eqn:prob1} and  \eqref{eqn:prob2}
 for $n=1000$:
(a) Dense graph with $p=0.5$;
(b) sparse graph with $p=1$ and $\gamma=0.3$.
The color of the corresponding pixel is blue if $w_{ij}^n=1$
 and it is light blue otherwise.
\label{fig:5a2}}
\end{figure}

We consider the following three cases for $G_{n}$:
\begin{enumerate}
\setlength{\leftskip}{-1.8em}
\item[(i)]
Complete simple graph, which is deterministic dense with $p=1$;
\item[(ii)]
Random undirected dense graph in which $w_{ij}^n=1$ with probability
\begin{equation}
\Pset(j\sim i)=p,\quad
i,j\in[n];
\label{eqn:prob1}
\end{equation}
\item[(iii)]
Random undirected sparse graph in which $w_{ij}^n=1$ with probability
\begin{equation}
\Pset(j\sim i)=n^{-\gamma}p,\quad
i,j\in[n],
\label{eqn:prob2}
\end{equation}
where $\gamma\in(0,\tfrac{1}{2})$.
\end{enumerate}
Note that $\alpha_n^{-1}=n^\gamma>1$.
We specifically take $p=0.5$ in case~(ii) and $p=1$ and $\gamma=0.3$ in case~(iii).
Similar numerical simulation results when the natural frequencies are deterministic
 were given in Section~3 of \cite{IY23}, as stated in Section~4.
Figures~\ref{fig:5a1} and \ref{fig:5a2},
 respectively, provide numerically computed samples
 of the natural frequencies and weight matrices
 for both the random undirected dense and sparse graphs,
 where $n=1000$, $p=0.5$ or $1$ and $\gamma=0.3$.
 
\begin{figure}[t]
\begin{center}
\includegraphics[scale=0.275]{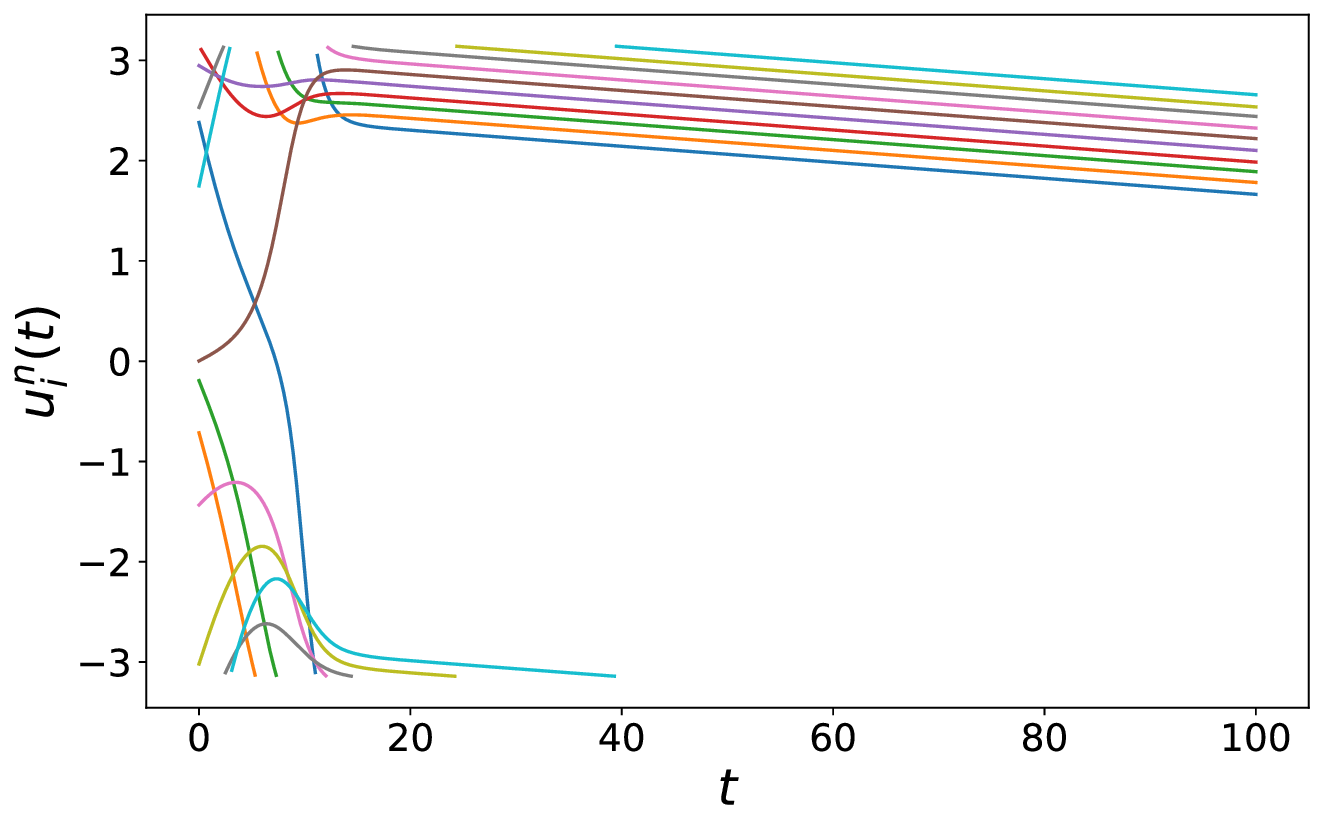}\\[-1ex]
{\footnotesize(a)}\\[2ex]
\end{center}
\begin{minipage}[t]{0.495\textwidth}
\begin{center}
\includegraphics[scale=0.275]{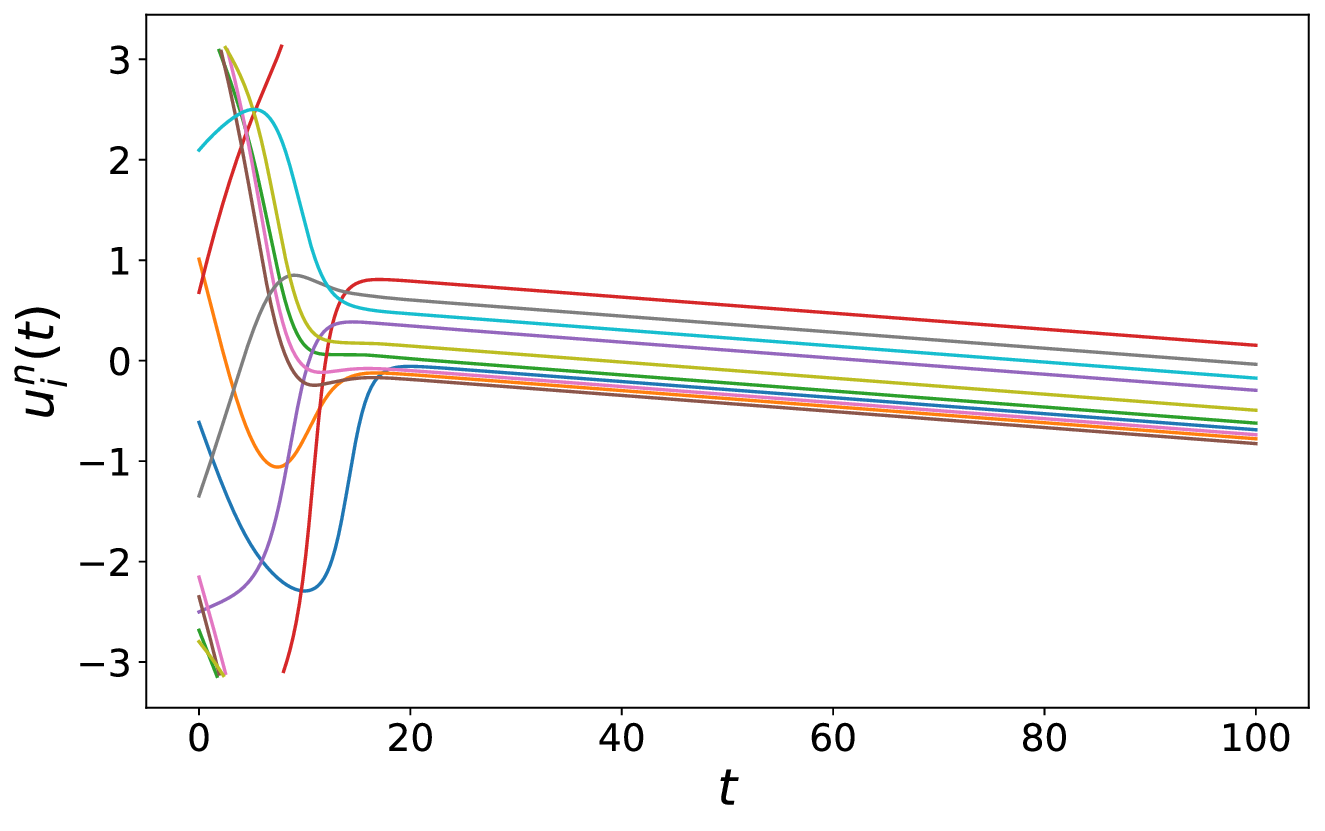}\\[-1ex]
{\footnotesize(b)}
\end{center}
\end{minipage}
\begin{minipage}[t]{0.495\textwidth}
\begin{center}
\includegraphics[scale=0.275]{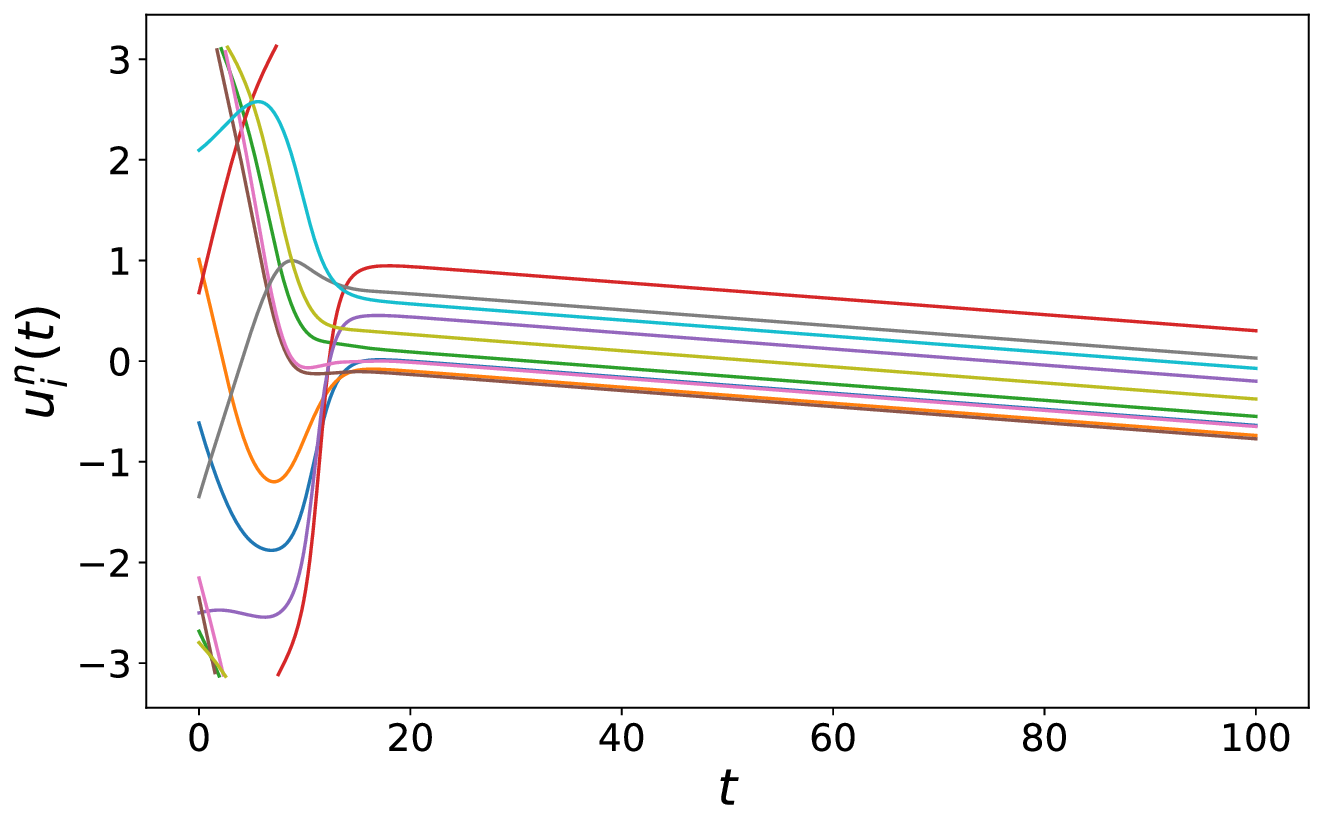}\\[-1ex]
{\footnotesize(c)}
\end{center}
\end{minipage}
\caption{Numerical simulation results for the KM \eqref{eqn:dsys} with $n=1000$:
(a) $(K,a)=(1,1)$ in case~(i); (b) $(K,a,p)=(2,1,0.5)$ in case (ii);
(c) $(K,a,p,\gamma)=(1,1,1,0.3)$ in case~(iii).
The time history of every 100th node
 (from 50th to 950th) is plotted with a different color.
Note that $p=1$ in case~(i) and $\gamma$ has a meaning  only in case~(iii).
\label{fig:5b}}
\end{figure}

We carried out numerical simulations for the KM \eqref{eqn:dsys}
 with $n=1000$ in cases~(i)-(iii), using the DOP853 solver,
 for which an exposition of the used method is found in \cite{HNW93}.
The initial values $u_i^n(0)$, $i\in[n]$,
 were independently randomly chosen on $[-\pi,\pi]$.

Figures~\ref{fig:5b}(a), (b) and (c)
 show the time-histories of every $100$th node (from 50th to 950th).
 for cases~(i), (ii) and (iii), respectively.
Here $u_i^n(t)$, $i=50,150,\ldots,950$, are plotted in the range of $[-\pi,\pi]$.
The value of $pK/a=1$ is larger than the critical value $2/\pi$,
 which is given in Theorem~\ref{thm:2f}.
We observe that the responses $u_i^n(t)$, $i\in[n]$,
 converge to the completely synchronized states in these figures,
 as predicted by Theorem~\ref{thm:4a}.
We also see that they exhibit rotational motions with a constant speed
 since the most probable value $\Omega$ in \eqref{eqn:Omega}
 is not zero by the asymmetric property of $\omega_i^n$, $i\in[n]$.
We note that the same sample of the natural frequencies was used
 in Fig.~\ref{fig:5b}(a)-(c).

\begin{figure}[t]
\begin{center}
\includegraphics[scale=0.28]{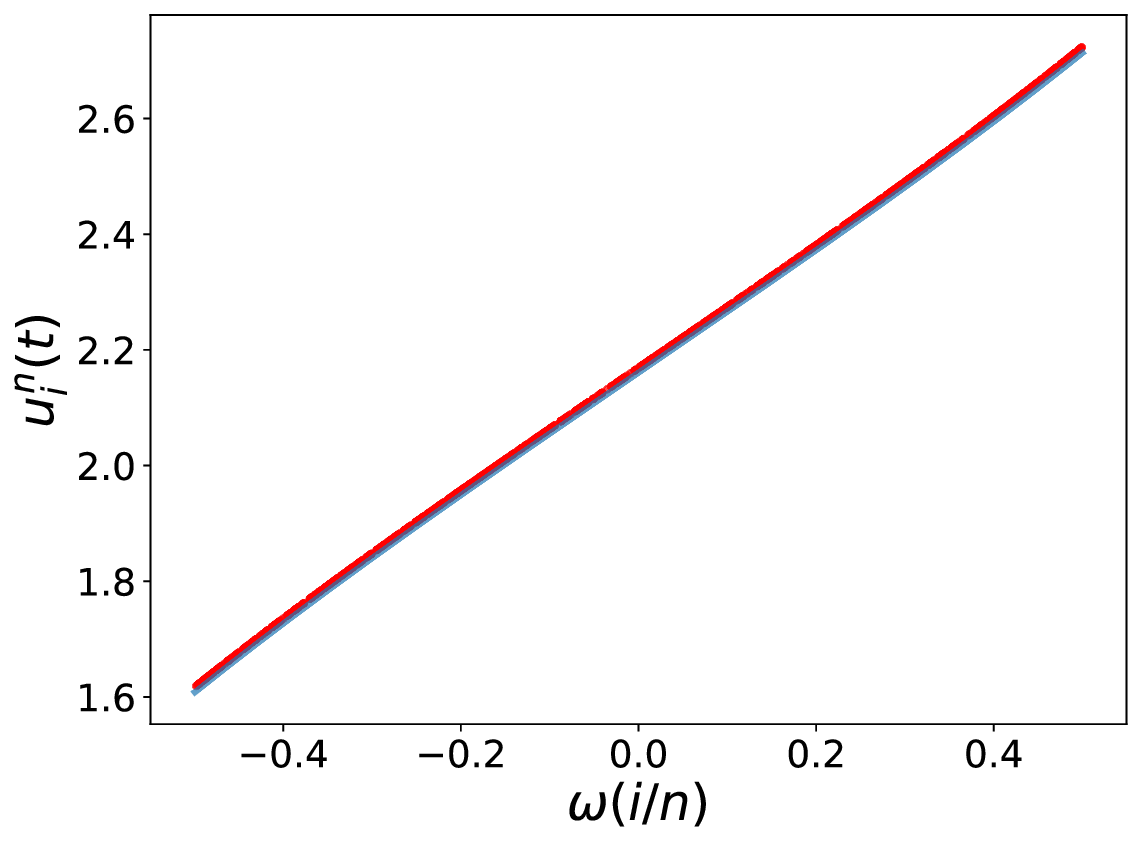}\\[-1ex]
{\footnotesize(a)}\\[2ex]
\end{center}
\begin{minipage}[t]{0.495\textwidth}
\begin{center}
\includegraphics[scale=0.28]{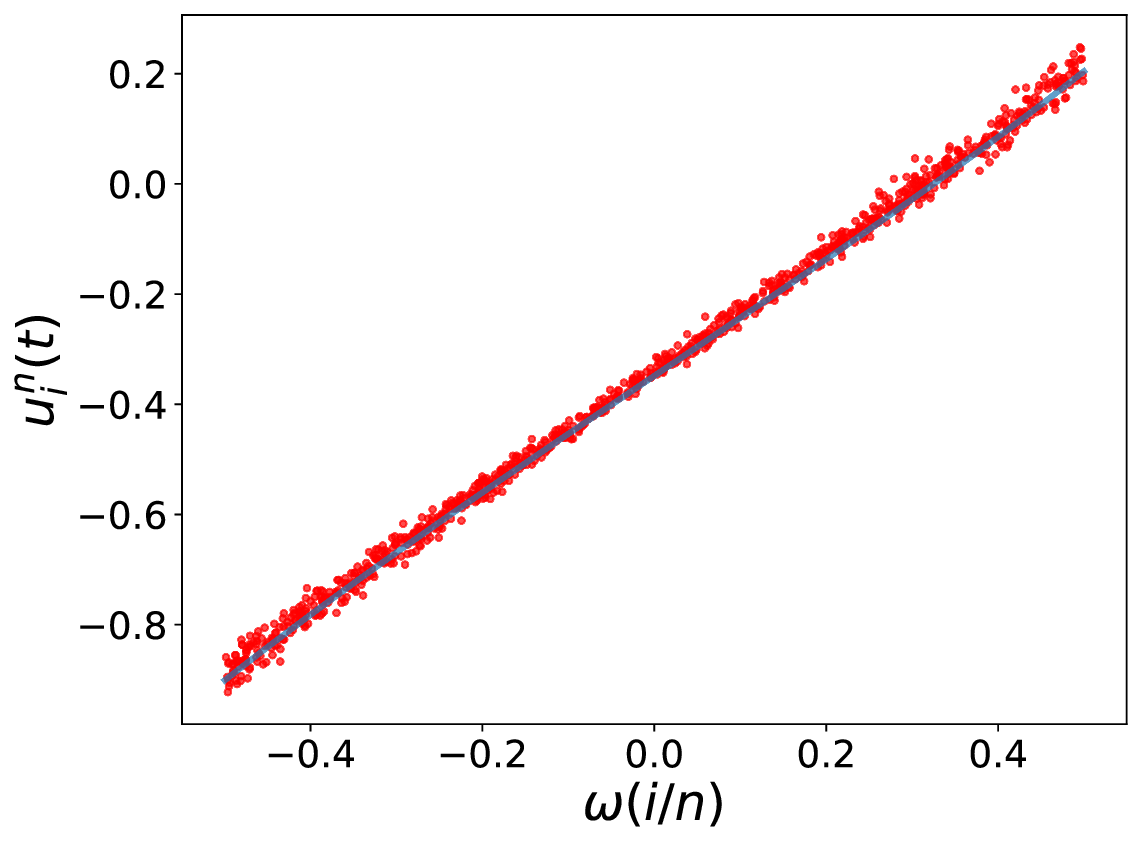}\\[-1ex]
{\footnotesize(b)}
\end{center}
\end{minipage}
\begin{minipage}[t]{0.495\textwidth}
\begin{center}
\includegraphics[scale=0.28]{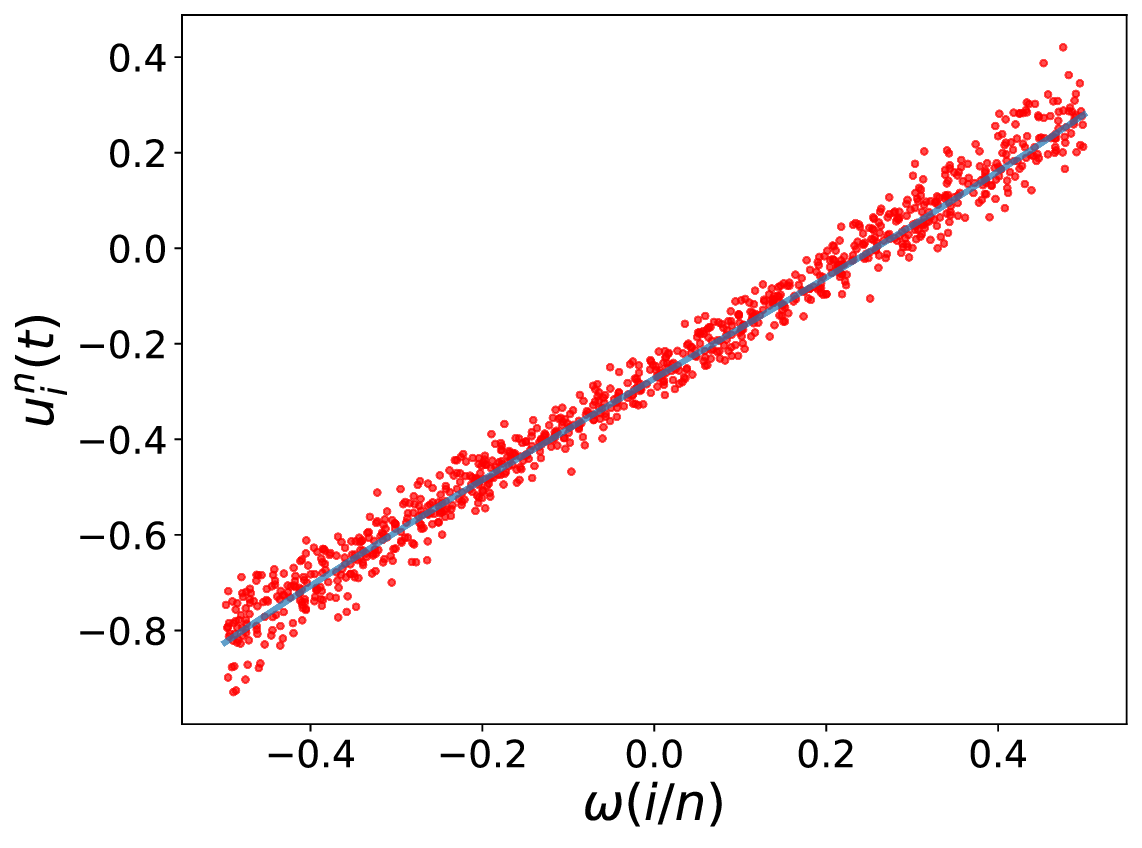}\\[-1ex]
{\footnotesize(c)}
\end{center}
\end{minipage}
\caption{Steady states of the KM \eqref{eqn:dsys} with $n=1000$:
(a) $(K,a)=(1,1)$ in case~(i); (b) $(K,a,p)=(2,1,0.5)$ in case (ii);
(c) $(K,a,p,\gamma)=(1,1,1,0.3)$ in case~(iii).
Here $u_i^n(t)$, $i\in[n]$, with $t=100$ are plotted as red dots.
The blue line represents the continuous stationary solution \eqref{eqn:csol0}
 to the CL \eqref{eqn:csys},
 where $\theta$ is the mean values of $u_i^n(100)$, $i\in[n]$.
\label{fig:5c}}
\end{figure}
 
In Figs.~\ref{fig:5c}(a), (b) and (c), $u_i^n(t)$, $i\in[n]$, with $t=100$,
 which may be regarded as the steady states from the results of Fig.~\ref{fig:5b},
 are plotted for cases~(i), (ii) and (iii), respectively.
Here the same initial condition, natural frequencies,
and values of $K$, $a$, $p$ and $\gamma$ as in Fig.~\ref{fig:5b} were used.
The blue line in each figure
 represents the continuous stationary solution \eqref{eqn:csol0} to the CL \eqref{eqn:csys},
 where  $\theta$ is the mean values of $u_i^n(100)$, $i\in[n]$, i.e.,
\[
\theta=\frac{1}{n}\sum_{i=1}^nu_i^n(100),
\]
for the corresponding result of Fig.~\ref{fig:5c}.
See Table~\ref{tbl:5a} for the estimated value of $\theta$ for each figure.
The agreement between the numerical results
 and theoretical predictions by Theorem~\ref{thm:4a}
 is fine for case~(i) in Fig.~\ref{fig:5c} (a)
 and good for case~(ii) in Fig.~\ref{fig:5c}(b)
 although some fluctuations due to randomness are found in the latter.
However, it is not so good for case~(iii) in Fig.~\ref{fig:5c}(c).
The reason for this disagreement is considered to be that
 the node number $n=1000$ is not enough
 for approximation of the KM \eqref{eqn:dsys} by the CL \eqref{eqn:csys} in case~(iii).
Such an observation for the difference between random sparse graphs
 and deterministic or random dense ones was given
 for the deterministic natural frequency cases \cite{IY23,Y24b,Y24c}.

\begin{table}[t]
\caption{
Mean value $\theta$ of $u_i^n(100)$,
 for the numerical results of Fig.~\ref{fig:5c}.
The numbers are rounded up to the fifth decimal point.
\label{tbl:5a}}
\begin{tabular}{c|c|c|c}
 & (a)  &  (b) & (c)\\
\hline
$\theta$ & $2.16107$  & $ -0.34812$ & $ -0.27273$\\
\end{tabular}
\end{table}

\begin{figure}[t]
\begin{center}
\includegraphics[scale=0.5]{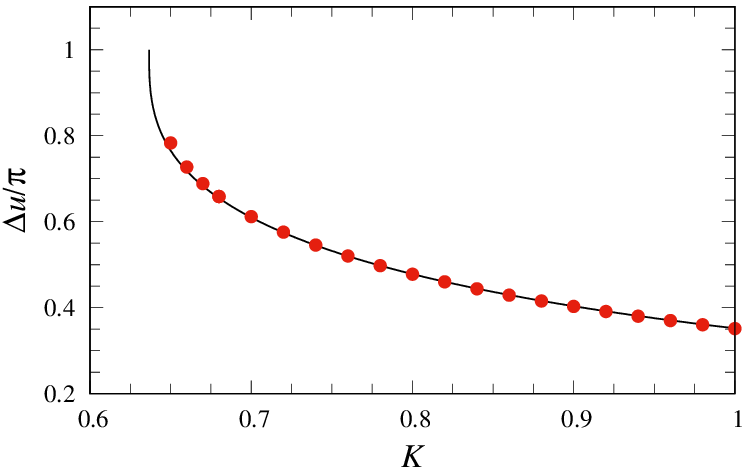}
\end{center}
\caption{Numerically computed bifurcation diagrams for the KM \eqref{eqn:dsys}
 with $n=1000$ in case~(i):
The ordinate $\Delta u$ represents the difference between $u_i^n(t)$ at the nodes 
 where $\omega_i^n$ is maximum or minimum when $t$ is sufficiently large.
The small red disks and black line represent numerical computations
 and theoretical predictions for the CL \eqref{eqn:csys} in Theorem~\ref{thm:4a}.
\label{fig:5d}}
\end{figure}

In Fig.~\ref{fig:5d}
 we give numerically computed bifurcation diagrams for the KM \eqref{eqn:dsys}
 with $n=1000$ in case~(i) when $K$ is taken as the control parameter.
The ordinate $\Delta u$ represents the difference
 between $u_i^n(t)$, $i=i_1$ and $i=i_2$, i.e.,
\[
\Delta u=u_{i_2}^n(t)-u_{i_1}^n(t),
\]
where $\omega_i^n$ is minimum and maximum when $i=i_1$ and $i_2$, respectively,
 for $t>0$ sufficiently large.
Numerical computations are plotted as small red disks,
 and the corresponding quantity for the continuous stationary solution \eqref{eqn:csol0}
 to the CL \eqref{eqn:csys}, i.e.,
\[
\Delta u=2\arcsin\left(\frac{a}{2KC}\right),
\]
as the black line for comparison.
Their agreement is not only qualitatively but also quantitatively fine.
At $K=0.64$,
 a synchronized motion was not observed in numerical simulations.

Thus, as predicted in Theorem~\ref{thm:4a}, we observe that
 the asymptotically stable family \eqref{eqn:csol0} of continuous stationary solutions
 to the CL \eqref{eqn:csys} with \eqref{eqn:lomega}
 behaves as if it is an asymptotically stable one
 in the KM \eqref{eqn:dsys} with random natural frequencies,
 under an appropriate random permutation.
So the method of CLs is still valid
 even when the natural frequencies are random.

\section*{Acknowledgments}
This work was partially supported by JSPS KAKENHI Grant Number JP23K22409.



\end{document}